\documentclass[draft]{amsart}
\usepackage{amssymb, enumerate, xspace, graphicx, url}
\usepackage{microtype}
\usepackage[latin1]{inputenc}
\usepackage[all]{xy}
\SelectTips{cm}{}

\numberwithin{equation}{section}

\numberwithin{subsection}{section}

\allowdisplaybreaks[1]

\newenvironment{enumeratea}
{\begin{enumerate}[\upshape (a)]}
{\end{enumerate}}

\newtheorem*{namedtheorem}{\theoremname}
\newcommand{\theoremname}{testing}

\newtheorem{theorem}{Theorem}[section]
\newtheorem{proposition}[theorem]{Proposition}
\newtheorem{proposition-definition}[theorem]
{Proposition-Definition}
\newtheorem{corollary}[theorem]{Corollary}
\newtheorem{lemma}[theorem]{Lemma}
\newcommand \fr{\operatorname{F}}
\theoremstyle{definition}
\newtheorem{definition}[theorem]{Definition}

\newtheorem{example}[theorem]{Example}

\theoremstyle{remark}

\DeclareMathOperator{\Lie}{Lie} \DeclareMathOperator{\W}{W}

 \newcommand\cB{\mathcal{B}}

\newcommand\cG{\mathcal{G}} 
 
 \newcommand\cL{\mathcal{L}}

\newcommand\GG{\mathbb{G}}

 \newcommand\ZZ{\mathbb{Z}}

\newcommand\rma{\mathrm{a}} 
 \newcommand\rmd{\mathrm{d}}

\newcommand\rmm{\mathrm{m}}

\newcommand\rmu{\mathrm{u}}

\newcommand\arr{\ifinner\to\else\longrightarrow\fi}

\newcommand\arrto{\ifinner\mapsto\else\longmapsto\fi}

\newcommand \In{\subseteq}

\renewcommand\H{\operatorname{H}}

\newcommand\eqdef{\overset{\mathrm{\scriptscriptstyle def}} =}

\renewcommand\th{^\text{th}}

\def\displaytimes_#1{\mathrel{\mathop{\times}\limits_{#1}}}

\def\displayotimes_#1{\mathrel{\mathop{\bigotimes}\limits_{#1}}}

 \DeclareMathOperator{\Gl}{Gl}

\newcommand\aut{\operatorname{Aut}}

\newcommand\spec{\operatorname{Spec}}

\newcommand\id{\mathrm{id}}

\newcommand\pr{\operatorname{pr}}

\newdir{ >}{{}*!/-5pt/@{>}}

\newcommand\doublelong[2]{\mathbin{\xymatrix{{}\ar@<3pt>[r]^{#1}
\ar@<-3pt>[r]_{#2}&}}}

\newlength{\ignora}

\newcommand{\catsch}[1]{(\mathrm{Sch}/#1)}

\newcommand{\mmu}{\boldsymbol{\mu}}

\newcommand{\aalpha}{\boldsymbol{\alpha}}

\newcommand{\ed}{\operatorname{ed}}

\newcommand{\lie}{\operatorname{Lie}}

\newcommand{\trdeg}{\operatorname{tr\,deg}}

\newcommand{\gm}{\GG_{\rmm}}

\newcommand{\ga}{\GG_{\rma}}

\begin{document}

\title[On the essential dimension of infinitesimal group schemes]{On the essential dimension\\of infinitesimal group schemes}

\author{Dajano Tossici}
\author[Angelo Vistoli]{Angelo Vistoli}

\address{Scuola Normale Superiore\\Piazza dei Cavalieri 7\\
56126 Pisa\\ Italy}
\email[Tossici]{dajano.tossici@sns.it}
\email[Vistoli]{angelo.vistoli@sns.it}

\date{September 4, 2010}

\maketitle

\begin{abstract}
We discuss essential dimension of group schemes, with particular attention to infinitesimal group schemes. We prove that the essential dimension of a group scheme of finite type over a field $k$ is greater than or equal to the difference between the dimension of its Lie algebra and its dimension. Furthermore, we show that the essential dimension of a trigonalizable group scheme of length $p^{n}$ over a field of characteristic~$p>0$ is at most~$n$. We give several examples.

\end{abstract}

\section{Introduction}

The notion of essential dimension of a finite group over a field $k$ was introduced by Buhler and Reichstein (\cite{BR}). It was
later extended to various contexts. First Reichstein
generalized it to linear algebraic groups (\cite{Re}) in
characteristic zero; afterwards Merkurjev gave a general definition for covariant functors from the
category of extension fields of the base field $k$ to the category of sets
(\cite{BF}). Brosnan, Reichstein and Vistoli (\cite{BRV})
studied the essential dimension of  algebraic stacks, a general class which
includes almost all the examples of interest.

Important results  on the essential dimension of finite groups in
characteristic $0$ have been proved by Florence (\cite{Flo}) and
Karpenko and Merkurjev (\cite{KM}). The essential dimension of
finite groups in positive characteristic has been studied by Ledet
\cite{Le}. As to higher dimensional groups, there are works of
Reichstein and Youssin (\cite{Rei-You}), Chernousov and Serre
(\cite{Che-Ser}), Gille and Reichstein \cite{Gil-Rei}, Brosnan,
Reichstein and Vistoli \cite{BRV2} on algebraic groups, Brosnan
\cite{Bro} on abelian varieties over $\mathbb{C}$, and Brosnan and
Shreekantan \cite{brosnan-shreekantan08} on abelian varieties over
number fields.
The case of non-smooth group schemes (necessarily in positive
characteristic) has not been investigated.

Let $k$ be a field of characteristic $p > 0$. For any scheme $X$ over $k$  let  $X^{(p)}$ denote the scheme $X\times_{\spec k}\spec k$,
   where $\spec k$ is endowed with the structure of $k$-scheme via
   the absolute Frobenius, which at the level of algebras is given by $a \arrto a^{p}$. We will endow $X^{(p)}$ with the structure of a $k$-scheme given by the projection on the second factor. Moreover we consider the relative Frobenius $\fr: X\arrto X^{(p)}$, which is the $k$-morphism induced by the absolute Frobenius on $X$.
We observe that if the $k$-scheme $X$ is in fact defined over $\mathbb{F}_p$ then $X^{(p)}$ is isomorphic to $X$. When this is the case we replace $X^{(p)}$ with $X$.

 If $m$ is a positive integer, we denote as usual by $\aalpha_{p^{m}}$ the kernel of the $m$-th power of the relative Frobenius $F^{m}\colon \ga \arr \ga$ over $k$, while $\mmu_{p^{m}}$ is the kernel of the ${p^{m}}\th$-power homomorphism $\gm \arr \gm$.
It is certainly known to
experts that the essential dimensions of $\aalpha_{p^m}$ and
$\mmu_{p^m}$
are $1$, and
that the essential dimension of $\mmu_{p^{m}}^n$ is $n$.  To the
authors' knowledge nothing else was known. The purpose of this paper
is to throw some light on this subject. In particular we will focus
on the essential dimension of infinitesimal (i.e. connected and
finite) group schemes. This could give, for instance, some
information about the essential dimension of Abelian varieties over
a field $k$ of characteristic $p > 0$ with $p$-rank equal to zero.
The $p$-rank $f$ of an Abelian variety $A$ is defined by
$A[p](\bar{k})\simeq (\ZZ/p\ZZ)^f$, where $A[p]$ is the group scheme
of $p$-torsion points and $\bar{k}$ is an algebraic closure of $k$.
If an Abelian variety $A$ has positive $p$-rank, its essential dimension is conjecturally infinite, since for any $n > 0$, $A$ contains, over $\bar{k}$, the group scheme $\ZZ/p^n \ZZ$, whose essential dimension is conjecturally equal to $n$ (see Lemma~\ref{lem1} and Example \ref{ex4}).

If $G$ is a group scheme of
finite type over a field $k$, its essential dimension $\ed_{k}G$
is the essential dimension of the stack $\mathcal{B}_kG$. Let us
recall the definition. Let $k$ be a field, and $G$ be a
group scheme of finite type over $k$. If $X$ is a  an algebraic space, over $k$, a
$G$-torsor on $X$ is  an algebraic space $P$ over $k$ with a right action of $G$, with a
$G$-invariant morphism $P \arr X$, such that fppf locally on $X$
the scheme $P$ is $G$-equivariantly isomorphic to $X\times_{\spec
k}G$. We recall that if $X$ is a scheme and $G$ is affine then any $G$-torsor over $X$ is in fact a scheme (\cite[Theorem III 4.3]{Mi}). Isomorphism classes of $G$-torsors on $X$ form a pointed set
$\H^{1}(X, G)$; if $G$ is commutative, then $\H^{1}(X, G)$ is a
group, and coincides with the cohomology group of $G$ in the fppf
topology.

\begin{definition}
Let $G$ be a group scheme of finite type over a field $k$. Let
$k\In K$ be an extension field and $[\xi]\in \H^1(\spec(K),G)$ the
class of a $G$-torsor $\xi$. Then the essential dimension of $\xi$
over $k$, which we denote by $\ed_k\xi$, is the smallest
nonnegative integer $n$ such that there exists a subfield $L$ of
$K$ containing $k$, with $\trdeg(L/k) \leq n$  such that $[\xi]$ is
in the image of the morphism $\H^1(\spec(L),G)\arr
\H^1(\spec(K),G)$.

The essential dimension of $G$ over $k$, which we denote by
$\ed_k G$, is the supremum  of $\ed_k \xi$, where $K/k$ ranges
through all the extension of $K$, and $\xi$ ranges through all the
$G$-torsors over $\spec(K)$.

The essential dimension  of $\xi$ at a prime integer $p$, which we
denote by $\ed_k(\xi; p)$, is  the minimal value of
$\ed_k(\xi_{K'})$, as $K'$ ranges over all finite field extensions
$K'/K$ such that $p$ does not divide the degree $[K' {:} K]$ and
$[\xi_{K'}]$ is the image of $[\xi]$ in $\H^1(\spec(K'),G)$.
Finally the essential dimension $\ed_k(G; p)$ is  defined as the
supremum value of $\ed_k(\xi; p)$, as $K$ ranges over all field
extensions of $k$ and $\xi$ ranges over $\H^1(\spec(K),G)$.
\end{definition}

It is clear from definition that $\ed_k G\ge \ed_k(G;p)$ for any
integer $p$. If $G$ is smooth, then $G$-torsors are locally
trivial in the étale topology, and our definition coincides with
that of Berhuy and Favi (\cite{BF}); in the general case, to get
meaningful results one needs to use the fppf topology. For
example, if $G$ is an infinitesimal group scheme, the $G$-torsors
over a reduced scheme that are locally trivial in the étale
topology are in fact trivial.

Our first result is a general lower bound for the essential dimension.

\begin{theorem}\label{thm1}
Let $G$ be a group scheme of finite type over a field $k$ of
characteristic $p\ge 0$. Then
   \[
   \ed_{k} (G;p) \geq \dim_{k} \lie G - \dim G\,.
   \]
\end{theorem}

We also have a fairly general upper bound. Let us recall the definition of a trigonalizable group scheme.

\begin{definition}
Let $G$ be an affine group scheme of finite type over a field $k$. We say that
$G$ is \emph{trigonalizable} if it has a normal unipotent subgroup scheme $U$ such that $G/U$ is diagonalizable.
\end{definition}

The name is justified by the fact that, every trigonalizable group scheme over $k$, is a subgroup scheme of the group scheme of upper triangular $n \times n$ matrices over $k$, for some $n$. Any affine commutative group scheme over an algebraically closed field is trigonalizable (see
\cite[Theorem IV \S 3 ,1.1]{DG}).

\begin{theorem}\label{thm2}
Let $G$ be a finite trigonalizable group scheme over a field of
characteristic $p > 0$, of order $p^{n}$. Then $\ed_k G \leq n$.
\end{theorem}

For constant $p$-group schemes (which are unipotent) the above result has already been proved by Ledet (\cite{Le}).

The second and the third sections are devoted to the proofs of these
two theorems. In the last section, combining the lower and upper
bounds above, we calculate the essential dimension of some classes
of infinitesimal group schemes. In particular we prove that the
essential dimension of a trigonalizable group scheme of height $\le
1$, i.e. such that the Frobenius $\fr$ is trivial on it, is equal to
the dimension of its Lie algebra (Corollary~\ref{Cor1}). Example~\ref{counterexample}, due to R.~Lötscher, M.~MacDonald, A.~Meyer and Z.~Reichstein,  shows that there exist  infinitesimal group schemes with essential dimension strictly larger than the dimension of its Lie algebra. However, we don't have similar examples over an algebraically closed field. In Example~\ref{ex:test-example} we propose a class of a commutative unipotent group scheme whose essential dimensions we are unable to determine, which should be an important test case to determine whether it is reasonable to conjecture that equality holds for trigonalizable group schemes.

\subsection*{Acknowledgements} We would like to thank P.~Brosnan, F.~Oort and Z.~Reichstein for useful comments and corrections to the first version of this paper. We are also in debt with the referee for his very detailed and useful report.

\section{The proof of Theorem~\ref{thm1}}
First we first state a well known fact.

\begin{lemma}[{\cite[Corollary~4.3]{merkurjev-ed}}]\label{lem1}
If $G$ is a subgroup scheme of a  group scheme $H$, then

\begin{enumeratea}

\item $\ed_{k} G +\dim G\leq \ed_k H + \dim H$, and

\item $\ed_{k} (G;p) +\dim G\leq \ed_k (H;p) + \dim H$ for any
prime $p$.

\end{enumeratea}
\end{lemma}

Now we prove the Theorem. If the characteristic of $k$ is $0$,
then $G$ is smooth and there is nothing to prove. Suppose that the
characteristic of $k$ is $p
> 0$. Since the essential dimension does not increase after a base
change (\cite[Proposition 1.5]{BF}), we may assume that $k$ is  algebraically
closed.

   Let $G_{1}$ be the kernel of the relative Frobenius map $\fr: G \arr G^{(p)}$; then $G_{1}$ is an
infinitesimal group scheme, and $\lie G_{1} = \lie G$. Since by
Lemma \ref{lem1} we have $\ed_k (G_{1};p) \leq \ed_k (G;p) + \dim
G$, it is sufficient to show that $\ed_k (G_{1};p) \geq \dim_k
\lie G_{1}$; in other words, we may assume that $G = G_{1}$, i.e.,
by definition, $G$ has height at most~$1$.


Let $G$ act freely on an open subscheme $X$ of a representation of
$G$. If $K$ is the function field of the quotient $X/G$ and $E$ is
the function field of $X$, then we have a $G$-torsor
$\spec E \arr \spec K$. Set $n \eqdef \dim_{k} \lie G$. Let $K'/K$
be any extension of degree coprime to $p$. We need to prove that
for any such $K'$ then the essential dimension of the $G$-torsor
$\spec (E\otimes_K K') \arr \spec K'$ is at least $\dim_{k}\lie G$. Since the extension $K'/K$ is separable, while $E/K$ is purely inseparable, we have that $E\otimes_K K'$ is again a
field. Suppose that the $G$-torsor is defined over an extension
$L$ of $k$ contained in $K'$; we need to show that the
transcendence degree of $L$ over $k$ is at least~$n$. Let $\spec R
\arr \spec L$ be the $G$-torsor yielding $\spec (E\otimes_K K')
\arr \spec K'$ by base change. Clearly $R$ is a field.
By the definition of the Frobenius we have that the diagram
   \[
   \xymatrix{
   G\times \spec(R)\ar[d]^{\fr\times \fr}\ar[r]^{\mu} &
  \spec(R)\ar^{\fr}[d]\\
  G^{(p)}\times \spec(R)^{(p)}\ar[r]^{\mu}& \spec(R)^{(p)}}
   \]
   commutes, where $\mu$ is the action of $G$ on $\spec(R)$.
   Since $\fr$ is trivial on $G$ it follows, by the above diagram,  that
the Frobenius $\fr:\spec(R)\arr \spec(R)^{(p)}$ is
$G$-equivariant, where we consider the trivial action on the
target. Therefore  $\fr: \spec(R)\arrto \spec(R)^{(p)}$ factorizes through $\spec(R)\arr
\spec(R)/G=\spec(L)$. But the degree of the Frobenius $\fr:
\spec(R)\arr \spec(R)^{(p)}$ is equal, by the next lemma, to
$p^{d}$, where $d\eqdef\trdeg_k R$. While the degree of
$\spec(R)\arr \spec(L)$ is equal to the order of $G$, which is
$p^n$ with $n=\dim_k(\Lie G)$. Hence we have $\trdeg_k R=\trdeg_k
L\ge \dim_k(\Lie G)$, as wanted.

\begin{lemma}\cite[Corollary 3.2.27]{liu}
Let $R$ be a finitely generated extension of transcendence degree
$d$ of a perfect field $k$ of characteristic $p > 0$. Then the relative
Frobenius $\fr:\spec(R) \arr \spec(R)^{(p)}$ is a finite morphism
 of degree $p^{d}$.
\end{lemma}
%

\section{The proof of Theorem~\ref{thm2}}

Let us start with stating a few Lemmas that we will use in the
proof. The first two are well known.

\begin{lemma}\cite[Proosition IV \S 2, 2.5]{DG}
\label{lem3} Let $G$ be a commutative unipotent group scheme over
a field $K$. Then there exists  a central decomposition series
   \[
   1 = G_{0} \subseteq G_{1} \subseteq \dots \subseteq G_{r} = G
   \]
of $G$, such that each successive quotient $G_{i}/G_{i-1}$ is a subgroup scheme of $\ga$.
\end{lemma}

\begin{lemma}\cite[Corollary IV \S 2, 2.6]{DG}\label{lem:unichange}
If $G$ is a group scheme over a field $K$ and $E$ is an extension
of $K$, then $\spec E\times_{\spec K}G \arr \spec E$ is unipotent
if and only if $G$ is unipotent. In particular, any twisted form
of a unipotent group scheme is unipotent.
\end{lemma}

\begin{lemma}\label{lem2}
Let $G$ be a commutative unipotent group scheme over a field $K$.
Then $\H^{i}(K, G) = 0$, for $i\ge 2$.
\end{lemma}

\begin{proof}
By Lemma~\ref{lem3}, we may assume that $G$ is a subgroup of
$\ga$. The quotient $\ga/G$ is isomorphic to $\ga$ (\cite[Proposition IV \S 2,
1.1]{DG}); then the result follows from the fact that $\H^{i}(K,
\ga) =  0$ for $i\ge 1$.
\end{proof}

Now we prove the key Lemma.
\begin{lemma}\label{lem:main}
Suppose that we have an extension
   \[
   1 \arr G_1 \arr G \arr G_2 \arr 1
   \]
of group schemes over $k$, where $G_1$ is a commutative unipotent
normal group subscheme of a group scheme $G$. Let $P \arr \spec K$
be a $G$-torsor over an extension $K$ of $k$. Then there exists an
intermediate extension $k \subseteq F \subseteq K$ and a twisted
form $\widetilde{G_1} \arr \spec F$ of $G_1$ over $F$, such that $P$
is defined over an intermediate extension of transcendence degree at
most $\ed_{k} G_2 + \ed_{F} \widetilde{G_1}$ over $k$.

Furthermore, if $G_1$ is central in $G$, then $\widetilde{G_1} =
\spec F \times_{\spec k}G_1$.
\end{lemma}

\begin{proof}
Consider the induced $G_2$-torsor $Q \eqdef P/G_1 \arr \spec K$.
There exists an intermediate extension $k \subseteq E \subseteq K$
with $\trdeg_{k}E \leq \ed G_2$, such that $Q$ comes by base
change from a $G_2$-torsor $Q_{E} \arr \spec E$. I claim that
$Q_{E}$ lifts to a $G$-torsor $P_{E} \arr \spec E$.

To see this, consider the fppf gerbe of liftings $\cL \arr
\catsch{E}$. It is a fibered category over the category $\catsch{E}$ of $E$-schemes, whose objects over an
$E$-scheme $T \arr \spec E$ are $G$-torsors $P_{T} \arr T$, together
with isomorphisms of $G_2$-torsors $P_{T}/G_1 \simeq T \times_{\spec
E} Q_{E}$, or, equivalently, $G$-equivariant morphisms of
$T$-schemes $P_{T} \arr T \times_{\spec E} Q_{E}$. The arrows from
$P_{T} \arr T \times_{\spec E} Q_{E}$ to $P'_{T'} \arr T'
\times_{\spec E} Q_{E}$ are defined in the obvious way, as diagrams
   \[
   \xymatrix{
   P_{T}\ar[d]^{F}\ar[r] &
   T\times_{\spec E}Q_{E} \ar[r]^-{\pr_{1}} \ar[d]^{f\times\id} &
   T \ar[d]^{f}\\
   P_{T'}\ar[r] &
   T'\times_{\spec E}Q_{E} \ar[r]^-{\pr_{1}} &
   T'
   }
   \]
in which $F$ is $G$-equivariant. We need to show that $\cL$ has a global section over $\spec E$.

The action of $G$ on $G_1$ by conjugation descends to an action of
$G_2$ on $G_1$, since $G_1$ is commutative. Denote by
$\widetilde{G_1}$ the twisted form of the group scheme $\spec
E\times_{\spec{k}} G_1$ coming from the $G_2$-torsor $Q_{E} \arr
\spec E$; in other words, $\widetilde{G_1}$ is the quotient
$(Q_{E}\times_{\spec k}G_1)/G_2$, where $G_2$ acts on the right on
$Q_{E}$, and on $G_1$ by left conjugation. We claim that the
gerbe $\cL$ is banded by $\widetilde{G_1}$; that is, if $P_{T}
\arr T\times_{\spec E}Q_{E}$ is an object of $\cL(T)$, the
automorphism group $\aut(P_{T})$ is isomorphic to
$\widetilde{G_1}(T)$, and this isomorphism is functorial in $T$.
In fact, the twisted form $\widetilde{G}$ of $G$ obtained as the
quotient $P_{T}\times_{\spec k}G$ by the diagonal action of $G$ (where $G$ acts on itself by conjugation) is the automorphism group scheme of the $G$-torsor $P_{T}
\arr T$, and it contains $\widetilde{G_1}$ as the subgroup scheme
of automorphisms inducing the identity on $T\times_{\spec
E}Q_{E}$. Hence $\widetilde{G_1}$ is the automorphism group scheme
of the object $P_{T} \arr T\times_{\spec E}Q_{E}$, and this proves
the claim.

By \cite[Section IV 3.4]{Gir}, the equivalence classes of gerbes banded by
$\widetilde{G_1}$ are para\-metrized by the group $\H^{2}(K,
\widetilde{G_1})$; a gerbe corresponds to $0$, i.e., it is
equivalent to the classifying stack $\cB_{E}\widetilde{G_1}$, if and
only if it has a section. Now, by Lemma~\ref{lem:unichange},
$\widetilde{G_1}$ is unipotent; hence by Lemma~\ref{lem2} $\H^{2}(E,
\widetilde{G_1}) = 0$, so $\cL$ has a section, and the $G_2$-torsor
$Q_{E} \arr \spec E$ lifts to a $G$-torsor $P_{E} \arr Q_{E} \arr
\spec E$.

There is no reason why $\spec K \times_{\spec E}P_{E}$ should be
isomorphic to $P \arr \spec K$ as a $G$-torsor. However, by
construction, we have $P/G_1 \simeq \spec K\times_{\spec E}Q_{E}$.
Since as we just saw $\cL$ is a trivial gerbe banded by
$\widetilde{G_1}$, we have that $\cL$ is equivalent to the
classifying stack $\cB_{E}\widetilde{G_1}$, in such a way that the
lifting $P_{E} \arr Q_{E}$ corresponds to the trivial torsor
$\widetilde{G_1} \arr \spec E$. Then $P \arr Q$ gives an object of
$\cL(\spec K)$; this will be defined over a intermediate extension
$E \subseteq F \subseteq K$ of transcendence degree at most $\ed_{E}
\widetilde{G_1}$ over $E$. So over $F$ there will exist an object
$P_{F} \arr \spec F\times_{\spec E}Q_{E}$ which is isomorphic to $P
\arr Q$ when pulled back to $K$. Hence $P$ is defined over $F$, and,
since the transcendence degree of $F$ is at most equal to
$\ed_{k}G_2 + \ed_{E} \widetilde{G_1}$, the result follows.
\end{proof}

Now we are ready to prove the theorem.

First of all, suppose that $G$ is a diagonalizable group of order $p^{n}$; then $G$ is a product $\mmu_{p^{d_{1}}} \times \dots \times\mmu_{p^{d_{r}}}$
 for certain positive integers $d_{1}$, \dots,~$d_{r}$ with $d_{1} + \dots + d_{r} = n$. Then $G$ is a subgroup scheme of $\gm^{r}$,
 hence by the Lemma \ref{lem1} we have
$\ed G \leq r\leq n$.

Now, assume that $G$ is commutative unipotent of order $p^{n}$,
with $n
> 0$. Suppose that $G$ is a subgroup of $\ga$; then again by
Lemma~\ref{lem1} we have $\ed_{k}G \leq \dim \ga = 1 \leq n$, and
we are done.

If $G$ is not a subgroup scheme of $\ga$, we proceed by induction on
$n$. Assume that the result holds for all commutative unipotent
subgroup schemes of order $p^{m}$ with $m < n$. Let $G_1$ be a
nontrivial subgroup scheme that is a subgroup scheme of $\ga$ and
call $p^{m}$ its order. The group scheme $G_1$ exists by Lemma
\ref{lem3}. Then by Lemma~\ref{lem:main} we have
   \[
   \ed_k G \leq \ed_k (G/G_1) + \ed_{F}{G_1}_{F} \leq (n-m) + m = n;
   \] so the result holds for $G$.

Let $G$ be a trigonalizable infinitesimal group scheme of order
$p^{n}$. Once again, we proceed by induction on $n$. Let us
suppose that the result is true for all trigonalizable groups of
order $p^{m}$ with $m < n$. By definition, $G$ is an extension
   \[
   1 \arr G_{\rmu} \arr G \arr G_{\rmd}\arr 1\,,
   \]
where $G_{\rmu}$ is unipotent and $G_{\rmd}$ is diagonalizable. We
may assume that $G_{\rmu}$ is non-trivial, otherwise $G$ is
diagonalizable and we are done. If $G_1$ denotes the center of
$G_{\rmu}$, then $G_1$ is a nontrivial commutative unipotent normal
subgroup of $G$; set $G_2 \eqdef G/G_1$. Call $p^{m}$ the order of
$G_2$; by induction hypothesis we have $\ed_{k}G_2 \leq m$. Once
again using Lemma~\ref{lem:main}, we have that $\ed_{k}G \leq
\ed_{k}G_2 + \ed_{F}\widetilde{G_1}$ for some twisted form of $G_1$;
but by Lemma~\ref{lem:unichange} the group scheme $\widetilde{G_1}$
is still commutative unipotent, hence by the previous case
$\ed_{F}\widetilde{G_1} \leq n-m$, and we are done.

\section{The essential dimension of some group schemes}

The inequality of Theorem~\ref{thm1} is not an equality  in general, even for infinitesimal group schemes. Furthermore, the inequality of Theorem~\ref{thm2} does not hold in general for finite non-trigonalizable group schemes.

\begin{example}\label{counterexample}
The following is due to R.~Lötscher, M.~MacDonald, A.~Meyer and Z.~Reichstein \cite[Example~6.3]{lmmr}. Suppose that $\ell/k$ is a cyclic extension of degree~$p$ of fields of characteristic~$p > 0$; call $\sigma$ a generator for the Galois group $\operatorname{Gal}(\ell/k)$.  The automorphism group of $\mmu_{p^{2}}$ is a cyclic group of order $p(p-1)$; choosing an element of order~$p$ in this automorphism group gives an action of $\operatorname{Gal}(\ell/k)$ on $\mmu_{p^{2},\ell}$, which defines, by Galois descent, a form $G$ of $\mmu_{p^{2}}$ over $k$. This is an infinitesimal group scheme over $k$ of order~$p^{2}$, and $\dim_{k} \lie G = 1$; however, its essential dimension is $p$.
\end{example}

This example gives a strict inequality in Theorem~\ref{thm1}, and shows how the inequality of Theorem~\ref{thm2} does not hold in general. However, we don't know any examples of either phenomenon over an algebraically closed field.

In order to investigate this question, let us give the following definition.

\begin{definition}\label{def: almost special} Let $G$ be an affine group scheme of finite
type over a field $k$. If
\begin{equation*}
\ed_k G = \dim_k \Lie(G)-\dim G
\end{equation*}
then $G$ is called \textsl{almost special}.
\end{definition}

If $G$ is a linear smooth group scheme over a field $k$, then it is almost-special if and only if it has essential dimension~$0$. The term ``almost-special'' is justified by the following fact.

Recall that a linear group scheme $G$ over a field $k$ is called \emph{special} if every $G$-torsor defined over an extension of $k$ is trivial. Special groups were introduced by J.-P.~Serre in \cite{Ser} and studied by A.~Grothendieck in \cite{grothendieck-torsion}, and by many other authors since then.

\begin{proposition}
A smooth affine group scheme of finite type over a field is almost-special if and only if it is special.
\end{proposition}

This is immediate over an algebraically closed field, but seems to be new in this generality; we are grateful to the referee and to Z.~Reichstein (who also helped with the proof) for pointing this out.

\begin{proof} Let $G$ be a smooth affine group of finite type over a field $k$. It is obvious that if $G$ is special, then it is almost-special.

Assume that $G$ is almost-special. Let $V$ be a generically free finite-dimensional representation of $G$; by definition, there exists an non-empty open subscheme $U$ on which the action is free. Let $U/G$ be the quotient, which exist as an integral algebraic space of finite type over $k$; then the projection $U \arr U/G$ is a $G$-torsor. Let $\spec K$ be the generic point of $U/G$, and let $E \arr \spec K$ the pullback of the $G$-torsor $U \arr U/G$. Since $\ed_{k}G = 0$, then $E \arr \spec K$ is defined over an intermediate extension $k \subseteq \ell \subseteq K$ that is finite over $k$. Since $U/G$ is geometrically integral over $k$ we have that $k$ is algebraically closed in $K$, so $\ell = k$, and the torsor $E \arr \spec K$ is defined over $k$. We have a cartesian diagram
   \[
   \xymatrix{
   E \ar[r]\ar[d] & P\ar[d]\\
   \spec K \ar[r] & \spec k
   }
   \]
where $P \arr \spec k$ is $G$-torsor, and the morphism in the top row is $G$-equivariant.

Let us show that $G$ is connected. Let $G^{0}$ be the connected component of the identity in $G$. The scheme $P/G^{0}$ is integral, since $P$ is integral, and is finite over $k$; hence it is the spectrum of a finite extension of $k$ contained in $K$. So $P/G^{0} = \spec k$, and $G = G^{0}$, as claimed.

If $k$ is finite, then it follows from a famous theorem of Steinberg \cite[Theorem~1.9]{steinberg} that every $G$-torsor over a finite field is trivial. Since $\ed_{k}G = 0$ we see that every $G$-torsor over an extension of $k$ descends to a finite field, and so it is trivial.

Next, assume that $k$ is infinite. The torsor $E \arr \spec K$ is versal (see \cite[Example 5.3]{GMS}), hence if we show that it is trivial, then every other torsor over an extension of $k$ will be trivial. The $G$-equivariant morphism $E \arr P$ gives a $G$-equivariant rational map $U  \dashrightarrow P$. After restricting $U$, we may assume that $U  \dashrightarrow P$ is defined everywhere.  Since $U$ is a non-empty open subset  of an affine space over an infinite field, we have $U(k) \neq \emptyset$ and thus $P(k) \neq \emptyset$. So the torsor $P \arr \spec k$ is trivial, and hence so is $E \arr \spec K$.
\end{proof}

%
%



We observe that if $G$ is almost special then
$\ed_k G = \ed_k(G;p)$, where $p$ is the characteristic of $k$. In
the following we give several examples of almost special group
schemes. Most of them are connected and trigonalizable. We do not
know whether all connected trigonalizable group scheme are almost
special. For diagonalizable group scheme this is true (see
Example~\ref{ex: ed of some group schemes}). In the unipotent case
this is open: in \ref{ex4} we discuss what we consider to be a key
example  to clarify this question.

It is easy to give examples of infinitesimal almost special group
schemes that are not trigonalizable: for example, if $\cG$ is a
special non-trigonalizable smooth group scheme (e.g. $\cG=\Gl_n$),
the kernel of its Frobenius ${}_{\fr} \cG$ is almost special, as it
follows from Theorem~\ref{thm1} and Lemma~\ref{lem1}.

Finally we remark that if
$G$ and $H$ are almost special
 then the product $G\times H$ is also almost special, and
 $\ed_k(G\times H) = \ed_k G+\ed_k H$.
%

The following strategy allows us to find some almost special
group schemes. If an infinitesimal group scheme $G$ can be
embedded in a special algebraic group scheme of dimension
$n\eqdef\dim_k \Lie(G)$ then, using Theorem \ref{thm1} and Lemma
\ref{lem1}, we can conclude that the essential dimension is
exactly $n$, i.e. $G$ is almost special. Here is an example when
this happens. Later in \ref{ex:test-example} we will see another
example in which this argument can not be applied.

\begin{example}\label{ex: ed of some group schemes}
If $n$ is a positive integer, we denote by $W_{n}$ the group scheme of truncated Witt vectors of length $n$ (see \cite[Chapter~2, \S~6]{Ser-local}, \cite[Chapter~3, \S~1]{Haz} and \cite[Chapter~5, \S~1]{DG}). Let $G$ be the group scheme
\begin{equation*}
\prod_{j=1}^{t} {}_{\fr^{m_j}} W_{n_j}\times \prod_{i=1}^s
\mmu_{p^{l_i}} ,
\end{equation*}
where ${}_{\fr^{m}} W_{n}$ is the kernel of the iterated Frobenius
$\fr^m:W_n\arr W_n$.
 Then $G$ is almost special, i.e.
\begin{equation}\label{eq: 1}
\ed_k G = s+\sum_{j=1}^{t}n_j=\dim_{k}\Lie G. \end{equation} Indeed,
by the remark above, it is enough to remark that $G$ is a
closed subgroup of $\prod_{j=1}^{t} W_{n_j} \times \GG_m^s$, which is special.
\end{example}

Another class of examples is given by the following Corollary.
\begin{corollary}\label{Cor1}
 Any trigonalizable group scheme $G$ of height
$\le 1$ and of finite type over $k$ is almost special.
\end{corollary}
\begin{proof}
If $G$ has order $p^n$ then, since $G$ has height $\le 1$, the
dimension of its Lie algebra is $n$.
 Therefore the result follows from Theorems \ref{thm1} and \ref{thm2}.
\end{proof}

As a consequence of  Theorem \ref{thm2} and  Lemma \ref{lem:main}
we prove the following result.

\begin{corollary}\label{cor2}
Let
   \[
   1 \arr G_1 \arr G \arr G_2 \arr 1
   \]
be an extension of group schemes over a field $k$ of
characteristic $p>0$, with $G_1$ unipotent commutative of order
$p^n$, then
$$
\ed_k G\le n +\ed_k G_2.
$$
\end{corollary}
%
%
%
\begin{proof}
Let $P \arr \spec K$ be a $G$-torsor over an extension $K$ of $k$.
Then, from Lemma \ref{lem:main}, there exists an intermediate
extension $k \subseteq E \subseteq K$ and a twisted form
$\widetilde{G_1}
\arr \spec E$ of $G_1$ over $E$, such that
$$\ed_k P\le
\ed_E \widetilde{G_1}+ \ed_k G_2.$$
By Lemma
\ref{lem:unichange}, $\widetilde{G_1}$ is unipotent. Hence  by
Theorem \ref{thm2} we have $\ed_E \widetilde{G_1} \le n$,
 and  the result follows.
\end{proof}

\begin{example}\label{ex 1}
In addition to the hypotheses of  Corollary \ref{cor2}, let us
suppose that $G_1$  is of height $\le 1$, $G_2$ is almost special and
   \[
   \dim_k\Lie(G)=\dim_k \Lie(G_1)+\dim_k \Lie(G_2).
   \]
From Corollary \ref{cor2} and Theorem
\ref{thm1} it follows that
$$\ed_k G = \ed_k G_1 + \ed_k G_2.$$
The hypothesis on the
dimension of the Lie algebra of $G$ is satisfied, for instance, if
the extension is split.

\end{example}

\begin{example}\label{ex:test-example}
We do not have examples of trigonalizable group
schemes that are not  almost special.  If  such an example exists  a candidate should be
the following.

Let us consider the kernel $G_{r,m,n}$ of the morphism $\fr^m
+V^r:\W_n\arr \W_n$, where $V$ is the Verschiebung (called ``shift''
in \cite{Ser-local} and ``décalage'' in \cite{DG}), $m>0$ and
$0<r<n$. This is a group scheme of order $p^{mn}$. The dimension of
its Lie algebra is $r$, and it is embedded in a special group
scheme of dimension $n$, therefore $r\le \ed_k G_{r,m,n} \le n$.

Now consider the case $r=1$.  Then we will prove that $G_{1,m,n}$ can not be
embedded in any special algebraic group scheme $\cG$ of dimension
$1$, so we can not use an argument as in the Example \ref{ex: ed of
some group schemes} to conclude that $\ed_k G_{1,m,n} = 1$.  To prove
this fact we can clearly suppose that $k$ is algebraically closed.
Now let us suppose that such a group $\cG$ exists. First of all we
can suppose it is connected. Secondly we observe that it should be
unipotent. Indeed, since it is special it should be smooth (see Theorem \ref{thm1}) and affine
(\cite[Theorem 1]{Ser});
therefore, if it was not unipotent, from \cite[Proposition IV \S 2, 3.11]{DG} we
conclude that it contains a subgroup scheme isomorphic to $\GG_{m}$.
This implies that the kernel of Frobenius, $ _{\fr} \cG$, is not
unipotent, since it contains a subgroup scheme isomorphic to
$\mmu_{p}$. But since the Lie algebra of $G_{1,m,n}$ is equal to the
Lie algebra of $\cG$ we have that $_{\fr}\cG=  { _{\fr}} G_{1,m,n}$,
which is a contradiction since $G_{1,m,n}$ is unipotent. But any
unipotent smooth and connected group scheme of dimension $1$ over a
perfect field is isomorphic to $\GG_a$ (\cite[Corollary IV \S 2, 2.10]{DG}).
And $G_{1,m,n}$ is not a subgroup scheme of $\GG_a$.

This example is also interesting since,  if $k$ is perfect, the
group scheme $G_{1,1,2}$ is the $p$-torsion group scheme of a supersingular
elliptic curve (i.e., an elliptic curve with $p$-rank equal to zero). The $p^n$-torsion group
scheme of the same curve has a decomposition series with quotients
isomorphic to $G_{1,1,2}$.

\end{example}

\begin{example}\label{ex4} In \cite{Le}, Ledet conjectured that the essential dimension of
$\ZZ/p^n\ZZ$ over a field $k$ of characteristic $p$ is $n$. The
conjecture has been proved (easily) for $n\le 2$. Since $\ZZ/p^n\ZZ$
is contained in the group scheme of Witt vectors of dimension $n$,
$\W_n$, and this group is special then $\ed_k(\ZZ/p^n\ZZ)\le n$. The
same result also follows from Theorem \ref{thm2}, which for constant
$p$-group schemes was already known. The open problem (for $n>2$) is
to prove the opposite inequality. We remark that if this conjecture
is true over an algebraic closure $\bar{k}$ of $k$ then $\ed_k G$
is equal to $n$ for any twisted form $G$ of $\ZZ/p^n\ZZ$. Indeed
such a group scheme would be unipotent (see Lemma \ref{lem3}) so, by
the Theorem \ref{thm2}, its essential dimension is smaller or equal
to $n$. On the other hand it is at least $n$, since the essential
dimension does not increase by base change (see \cite[Proposition 1.5]{BF}).

\end{example}

\bibliographystyle{amsalpha}
\bibliography{biblioessdimension2}

\end{document}